\documentclass[10pt]{amsart}

\usepackage{latexsym, amsmath,amssymb}

\renewcommand{\epsilon}{\varepsilon}
\newcommand{\newsection}[1]
{\subsection{#1}\setcounter{theorem}{0} \setcounter{equation}{0}
\par\noindent}

\newtheorem{theorem}{Theorem}

\newtheorem{lemma}[theorem]{Lemma}
\newtheorem{corr}[theorem]{Corollary}

\newtheorem{proposition}[theorem]{Proposition}
\newtheorem{deff}[theorem]{Definition}

\newcommand{\bth}{\begin{theorem}}
\newcommand{\ble}{\begin{lemma}}
\newcommand{\bcor}{\begin{corr}}

\newcommand{\bdeff}{\begin{deff}}

\newcommand{\bprop}{\begin{proposition}}
\newcommand{\ele}{\end{lemma}}
\newcommand{\ecor}{\end{corr}}
\newcommand{\edeff}{\end{deff}}

\newcommand{\eprop}{\end{proposition}}

\renewcommand{\Pi}{\varPi}

\renewcommand{\epsilon}{\varepsilon}

\newcommand{\K}{{\mathcal K}}
\newcommand{\R}{{\mathbb R}}

\newcommand{\la}{{\langle}}
\newcommand{\ra}{{\rangle}}
\newcommand{\cd}{{\,\cdot\,}}
\newcommand{\bdy}{{\partial\K}}
\newcommand{\ext}{{\R^n\backslash\K}}
\newcommand{\extfour}{{\R^4\backslash\K}}

\begin{document}

\title[Global existence of quasilinear wave equations]
{
Global existence for high dimensional quasilinear wave equations
exterior to star-shaped obstacles
}

\author{Jason Metcalfe}
\address{Department of Mathematics, University of North Carolina, Chapel Hill}
\email{metcalfe@email.unc.edu}

\author{Christopher D. Sogge}
\address{Department of Mathematics, Johns Hopkins University}
\email{sogge@jhu.edu}

\thanks{The authors were supported in part by the NSF}

\maketitle

\newsection{Introduction}
The purpose of this article is to study long time existence for high
dimensional quasilinear wave equations exterior to star-shaped
obstacles.  In particular, we seek to prove exterior domain analogs of
the four dimensional results of \cite{Hormander} where the
nonlinearity is permitted to depend on the solution not just its first
and second derivatives.  Previous proofs in exterior domains omitted
this dependence as it did not mesh well with the energy methods in use.  The main estimates
used in the proof are the variable coefficient localized energy
estimate of \cite{MS} as well as a constant coefficient variant of
this estimate which was developed in \cite{DMSZ}, \cite{FW}, and \cite{HMSSZ}.

Let us more specifically describe the problem at hand.  We fix a
bounded set $\K$ which has smooth boundary and is star-shaped with
respect to the origin.  Without loss of generality, we shall assume
that $\K\subset\{|x|<1\}$.  We then seek to solve the following boundary
value problem 
\begin{equation}\label{generalEquation}
  \begin{cases}
    \Box u = Q(u,u',u''),\quad (t,x)\in \R_+\times\ext,\\
    u(t,\cd)|_\bdy = 0,\\
    u(0,\cd)=f,\quad \partial_t u(0,\cd)=g
  \end{cases}
\end{equation}
where $\Box=\partial_t^2-\Delta$ is the d'Alembertian.  Here and
throughout, we use $u'=(\partial_tu,\nabla_x u)$ to denote the
space-time gradient.  The nonlinear
term $Q$ is smooth in its arguments and has the form
\begin{equation}\label{Q}
Q(u,u',u'')=A(u,u')+ 
B^{\alpha\beta}(u,u')\partial_\alpha\partial_\beta u
\end{equation}
with the convenient notation $\partial_0=\partial_t$.  Throughout this
paper, we shall utilize the summation convention where repeated
indices are summed.  Greek indices $\alpha, \beta, \gamma$ are summed from $0$ to the spatial
dimension $n$, while Latin indices $i, j, k$ are implicitly summed from $1$ to
$n$.  We shall reserve $\mu$, $\nu$, and $\sigma$ for multiindices.  Here $A$ is
taken to vanish to second order at $(u,u')=(0,0)$, and $B$ vanishes to
first order at the origin.  We also assume the symmetry condition
\begin{equation}
 \label{symmetry}
B^{\alpha\beta}(u,u')=B^{\beta\alpha}(u,u'),\quad 0\le \alpha,\beta\le n.
\end{equation}

To solve \eqref{generalEquation}, one must assume that the data
satisfy some compatibility conditions.  These are well known, and we
shall only tersely describe them.  A more detailed exposition is
available in, e.g., \cite{KSS2}.  We write $J_ku=\{\partial^\mu_x
u\,:\, 0\le |\mu|\le k\}$.  For any formal $H^m$ solution $u$, we
can write $\partial_t^k u(0,\cd)=\psi_k(J_kf,J_{k-1}g)$, $0\le k\le m$
for some compatibility functions $\psi_k$.  The compatibility
condition of order $m$ for data $(f,g)\in H^m\times H^{m-1}$ simply
requires that $\psi_k$ vanishes on $\bdy$ for $0\le k\le m-1$.  For
$(f,g)\in C^\infty$, we say that the data satisfy the compatibility
condition to infinite order if the above holds for all $m$.

Under these assumptions, we have small data global existence when $n\ge 5$.
\begin{theorem}\label{thm1.1}
Let $\K\subset \R^n$, $n\ge 5$ be a smooth, bounded, star-shaped
obstacle, and let $Q(u,u',u'')$ be as above.  Assume further that
$(f,g)\in C^\infty(\ext)$ vanish for $|x|>R$ for fixed $R$ and satisfy
the compatibility conditions to infinite order.  Then there is a
constant $\varepsilon_0$ and a positive integer $N$ so that if
$0<\varepsilon<\varepsilon_0$ and
\begin{equation}
  \label{dataSmallness}
\sum_{|\mu|\le N+1} \|\partial_x^\mu f\|_{L^2(\ext)} +
\sum_{|\mu|\le N} \|\partial_x^\mu g\|_{L^2(\ext)}\le \varepsilon,
\end{equation}
then \eqref{generalEquation} has a unique global solution $u\in C^\infty([0,\infty)\times\ext)$.
\end{theorem}

When the $Q(u,u',u'')=Q(u',u'')$, such global existence results have
previously been established in \cite{MS, MS2} for $n\ge 4$.  The
geometrical restrictions on $\K$ in \cite{MS2} are much less strict.
When $n\ge 7$, the proof in \cite{MS} can easily be adapted to prove
the theorem.  We shall not discuss the $n=5,6$ result further as it
follows from easy modifications of the arguments of \cite{DMSZ} or of
those in the sequel.

With a general nonlinearity such as above, in dimension 4, one expects
almost global existence as in H\"ormander~\cite{Hormander} for the boundaryless
case, and this was indeed proved in \cite{DMSZ}.  Similarly, in
three dimensions, based on the boundaryless results of
Lindblad~\cite{Lindblad}, one expects a lifespan $T_\varepsilon \sim
\varepsilon^{-2}$, and this was proved for star-shaped obstacles in
\cite{DZ}.  

The proofs of \cite{Hormander} and \cite{Lindblad} for long time
existence in the boundaryless case show an improved lifespan in $n=3,4$ if the additional
restriction 
\begin{equation}
 \label{restriction}
(\partial_u^2 A)(0,0,0)=0
\end{equation}
is imposed.
With this additional restriction, the lifespan bounds which are proved
are comparable to those which were previously available when the
nonlinearity was not permitted to depend on $u$.  That is, in three
dimensions, solutions exist almost globally, and in four dimensions,
there is global existence.

The exterior domain analog of this four dimensional global existence
is the primary result of this article.
\begin{theorem}
  \label{thm1.2}
Let $\K\subset \R^4$ be a smooth, bounded, star-shaped
obstacle, and let $Q(u,u',u'')$ satisfy \eqref{restriction} in
addition to the assumptions of Theorem \ref{thm1.1}.  Assume further that
$(f,g)\in C^\infty(\extfour)$ vanish for $|x|>R$ for fixed $R$ and satisfy
the compatibility conditions to infinite order.  Then there is a
constant $\varepsilon_0$ and a positive integer $N$ so that if
$0<\varepsilon<\varepsilon_0$ and
\begin{equation}
  \label{dataSmallness2}
\sum_{|\mu|\le N+1} \|\partial_x^\mu f\|_{L^2(\extfour)} +
\sum_{|\mu|\le N} \|\partial_x^\mu g\|_{L^2(\extfour)}\le \varepsilon,
\end{equation}
then \eqref{generalEquation} has a unique global solution $u\in C^\infty([0,\infty)\times\extfour)$.
\end{theorem}

While we have only stated Theorem \ref{thm1.1} and Theorem
\ref{thm1.2} for scalar equations, it is not difficult to extend these
results to systems and even multiple speed systems.  We do not expect
the restriction to star-shaped obstacles to be optimal but impose this
largely for simplicity of exposition.  One might fully expect similar
results to hold in domains similar to those addressed in \cite{MS3},
that is, any domain for which there is a sufficiently rapid decay of
local energy.  

We also note here that we only examine the case of
Dirichlet boundary conditions.  While Neumann boundary conditions were
permitted in \cite{HMSSZ}, they are more difficult to handle for
quasilinear equations.  First of all, even proving energy estimates
for small perturbations of the d'Alembertian in the exterior domain
requires additional assumptions.  In particular, one either needs to
assume a nonlinear compatibility condition which is akin to what
appears in \cite{MSS} or more generally take the boundary
condition to regard the conormal derivative as in \cite{Koch}.
Moreover, with Dirichlet boundary conditions and a star shaped obstacle, the boundary
terms which appear in the localized energy estimates (see Proposition
\ref{kss.prop} and \cite{MS}) have a favorable sign.  This is no longer the case
with Neumann boundary conditions.  By developing techniques that would
allow more general obstacle geometries, it is possible that one may
also permit Neumann boundary conditions, but we do not explore that here.

In hopes of making the arguments more transparent, we shall truncate
the nonlinearity at the quadratic level.  Since we are dealing with
small amplitude solutions, the higher order terms are better behaved,
and it is clear how to alter the proofs in the sequel to permit these
terms.  With such a truncation, we may now focus on 
\begin{equation}
  \label{trucatedEquation}
\Box u = a^\alpha u \partial_\alpha u + b^{\alpha\beta}\partial_\alpha u \partial_\beta u + A^{\alpha\beta}
u\partial_\alpha\partial_\beta u + B^{\alpha\beta\gamma}\partial_\alpha u \partial_\beta\partial_\gamma u
\end{equation}
with Dirichlet boundary conditions and small data.

Our proof is based on two key estimates.  The first is a localized
energy estimate which, beginning with \cite{KSS}, has played a key
role in nearly every proof of long time existence for wave equations
in exterior domains.  In one of its simplest forms, it states
\begin{equation}
  \label{locEnergy1}
\|\la x\ra^{-1/2-} w'\|_{L^2_{t,x}([0,T]\times \R^n)}\lesssim
\|w'(0,\cd)\|_2 + \int_0^T \|\Box w(s,\cd)\|_2\:ds
\end{equation}
in $\R_+\times\R^n$, $n\ge 3$.  We shall also utilize a version of
this for perturbations of the d'Alembertian which is from \cite{MS}.

The second of the key estimates is a variant on this.  It can be
thought of as a generalization of the main new estimate used in
\cite{DMSZ}.  It is also the $p=2$ version of the weighted Strichartz
estimates of \cite{FW}, \cite{HMSSZ}.  In $\R_+\times \R^4$, this
states that
\begin{equation}
  \label{wtStrich}
\||x|^{-1/2-\gamma} w\|_{L^2_{t,x}} \lesssim
\|w'(0,\cd)\|_{\dot{H}^{\gamma-1}} + \||x|^{-1-\gamma} \Box
w\|_{L^1_tL^1_rL^2_\omega},\quad 0<\gamma<1/2.
\end{equation}
It is the $\gamma=0$ version of this estimate which was used in
\cite{DMSZ} to prove almost global existence in $\extfour$ when
\eqref{restriction} is not assumed.  When $\gamma=0$, there is a
logarithmic blow up in $t$ in the estimate, and this logarithm
corresponds precisely to the exponential in the lifespan bound.

%%%%%%%%%%%%%%%%%%%%%%%%%%%%%%%%%%%%%%%%%%%%%%%%%%%%%%%%%%%%%%%%%%%%%%%%%%%%%%%%%%%%%%%%%%%%%%%%%%%%
%%%%%%%%%%%%%%%%%%%%%%%%%%%%%%%%%%%%%%%%%%%%%%%%%%%%%%%%%%%%%%%%%%%%%%%%%%%%%%%%%%%%%%%%%%%%%%%%%%%%
%%%%%%%%%%%%%%%%%%%%%%%%%%%%%%%%%%%%%%%%%%%%%%%%%%%%%%%%%%%%%%%%%%%%%%%%%%%%%%%%%%%%%%%%%%%%%%%%%%%%
\newsection{Main estimates on $\R_+\times \R^4$}
In this section, we gather the main boundaryless estimates.  These
estimates, for the most part, are not new.  In the sequel, we shall
apply a cutoff which vanishes near the boundary to the solution.  What
results solves a boundaryless wave equation to which these estimates
may be applied. 

The estimates which we explore here will be for linear wave equations
in $\R_+\times\R^4$.  We let $w$ solve
\begin{equation}
  \label{bdyless}
  \begin{cases}
    \Box w = F,\quad (t,x)\in \R_+\times\R^4,\\
    w(0,\cd)=w_0,\quad \partial_t w(0,\cd)=w_1.
  \end{cases}
\end{equation}

The first of these estimates is a localized energy estimate.  This has
become an increasingly standard tool in the study of nonlinear wave
equations.  In the next section, we shall also present a version of
this estimate which holds for perturbations of the d'Alembertian

\begin{proposition}
  \label{propLocalEnergy}
Let $w$ be a smooth solution to \eqref{bdyless} which vanishes for
large $|x|$ for each $t$.  Then, for any $T>0$, we have
\begin{equation}
  \label{locEnergy}
\|\la x\ra^{-1/2-} w'\|_{L^2_{t,x}([0,T]\times \R^4)} + \|\la
x\ra^{-3/2} w\|_{L^2_{t,x}([0,T]\times \R^4)} \lesssim
\|w'(0,\cd)\|_2 + \int_0^T \|\Box w(t,\cd)\|_2\:dt
\end{equation}
with constant independent of $T$.
\end{proposition}

This proposition, in fact, holds in any dimension $n\ge 3$, though in
$n=3$ the second term in the left side has a logarithmic divergence in $T$.
A proof can be found in \cite{MS}, though this
estimate did not originate there.  The interested reader can see the
references therein for a more complete history.  To prove the
estimate, one uses a positive commutator
argument with multiplier $f(r)\partial_r w +
\frac{3}{2}\frac{f(r)}{r}w$ where $f(r)=\frac{r}{r+R}$ to get the
estimate in a torus with radii $\approx R$.  Summing over such dyadic
radii yields the proposition.

The second estimate is from \cite{FW} and \cite{HMSSZ}.
\begin{proposition}
  \label{propWeightedStrichartz}
Let $w$ be a smooth solution to \eqref{bdyless}.  Then, for
$0<\gamma<\frac{1}{2}$, we have 
\begin{equation}
  \label{weightedStrichartz}
\||x|^{-\frac{1}{2}-\gamma} w\|_{L^2_{t,x}(\R_+\times\R^4)} \lesssim
\|w_0\|_{\dot{H}^\gamma(\R^4)} + \|w_1\|_{\dot{H}^{\gamma-1}(\R^4)} +
\||x|^{-1-\gamma} F\|_{L^1_tL^1_rL^2_\omega(\R_+\times\R^4)}.
\end{equation}
\end{proposition}

This proposition follows in the homogeneous case by interpolating a
trace lemma (on a sphere) and a variant of the localized energy estimate which
follows by applying Plancherel's theorem in the $t$-variable.  The
inhomogeneous estimate follows, in turn, by using the dual estimate to
the trace lemma which was applied.  This yields a much wider class of
estimates, which were dubbed weighted Strichartz estimates, for which
we have only stated the $p=2$ case.

The $\gamma=0$ variant of this estimate, which involves a logarithmic
blow-up in $T$ when applied on $[0,T]\times\R^4$, laid at the heart of
the proof of the four dimensional almost global analog of Theorem
\ref{thm1.1} in \cite{DMSZ}.  Here one alters \eqref{locEnergy} by
applying it to the Reisz tranforms of the solution.  A weighted
variant of Sobolev's lemma is applied to the resulting Sobolev norm
with negative index which contains the forcing term.

To obtain a boundaryless wave equation, one applies a cutoff to the
solution in the exterior domain.  In order to handle the resulting
commutator term, we
shall require a variant of these estimates that permits the forcing
term to be taken in $L^2_t$ provided that it is compactly supported in
the spatial variable, uniformly in $t$.  
\begin{proposition}
 \label{propCpctF}
Let $w$ be a smooth solution to \eqref{bdyless} with vanishing data ($w_0=w_1=0$).
% which vanishes for
%large $|x|$ for each $t$.  
Suppose that $F(t,x)=0$ for $|x|>2$.  Then,
\begin{equation}
  \label{cpctF}
\|\la x\ra^{-1/2-} w\|_{L^2_{t,x}([0,T]\times \R^4)}\lesssim
\|F\|_{L^2_{t,x}([0,T]\times \R^4)}.
\end{equation}
\end{proposition}

This proposition is from \cite{DMSZ} and in fact holds for $n\ge 3$.
The techniques described below, however, only work for $n\ge 4$.  The
$n=3$ case was presented in \cite{DZ} using ideas from \cite{KSS}.
When $w$ is replaced by $w'$ in the left side, this estimate follows
from the proof of \eqref{locEnergy} described above.  Rather than
applying the Schwarz inequality in $x$ and bounding the multiplier term using
the energy inequality, one instead applies Schwarz in $t$ and $x$
while appropriately introducing a weight so that the multiplier term
can be bootstrapped into the left side of the estimate.  In order to
obtain \eqref{cpctF}, we apply this to $\partial_j
(\Delta^{-1}\partial_j w)$.  Applying $\Delta^{-1}\partial_j$ to $F$
does not maintain the compact support, but the kernel is
$O(|x-y|^{-3})$ which remains sufficient to absorb the weight which
was introduced.

We will utilize one additional variant of the localized energy
estimates.  This one is obtained using techniques akin to those which
appeared in \cite{Hormander} and \cite{Lindblad}.
\begin{proposition}
  \label{propLindbladVariant}
Let $v$ be a smooth solution to 
\begin{equation}\label{divEquation}
  \begin{cases}
    \Box v=\sum_0^4 a_j \partial_j G,\quad (t,x)\in \R_+\times\R^4,\\
    v(0,\cd)=\partial_t v(0,\cd)=0.
  \end{cases}
\end{equation}
Then, 
\begin{equation}
  \label{lindbladVariant}
\|\la x\ra^{-1/2-\delta} v\|_{L^2_{t,x}([0,T]\times\R^4)} \lesssim
\|G(0,\cd)\|_{\dot{H}^{\delta-1}} + \int_0^T \|G(t,\cd)\|_2\:dt
\end{equation}
for $0<\delta<1/2$.
\end{proposition}

\begin{proof}
  We let $v_1$ solve $\Box v_1=G$ with vanishing data, and let $v_0$
  solve the homogeneous equation $\Box v_0=0$ with $v_0(0,\cd)=0$,
  $\partial_t v_0(0,\cd)=G(0,\cd)$.  Then, 
\[ v=\sum_0^4 a_j \partial_j v_1 - a_0 v_0.\]
Thus,
\[ \|\la x\ra^{-1/2-\delta} v\|_{L^2_{t,x}} \lesssim \|\la
x\ra^{-1/2-\delta}v_1'\|_{L^2_{t,x}} + \|\la
x\ra^{-1/2-\delta}v_0\|_{L^2_{t,x}}.\]
For the first term in the right side, we simply apply
\eqref{locEnergy}, and to the second term we apply \eqref{weightedStrichartz}.
\end{proof}

We end this section with our principal source of decay.  Here, as was
initiated in \cite{KSS}, we use the localized energy estimates so as
to obtain long time existence from decay in the $|x|$ variable rather
than decay in the $t$ variable, which is more standard in the
boundaryless case but much more difficult to prove when there is a
boundary.  The decay that we obtain is based on the vector fields that
generate translations and rotations.  To that end, we set
\[\{\Omega\} = \{\Omega_{jk}=x_j\partial_k - x_k\partial_j\},\quad 1\le
j<k\le 4\]
and
\[\{Z\} = \{\partial_\alpha, \Omega_{jk}\},\quad 0\le \alpha\le 4, \, 1\le
j<k\le 4.\]

\begin{lemma}
  \label{lemmaWeightedSobolev}
For $h\in C^\infty(\R^4)$ and $R>1$, 
\begin{equation}\label{weightedSobolev}
  \|h\|_{L^\infty(\{|x|\in [R,2R]\})} \lesssim R^{-3/2}
  \sum_{|\mu|\le 2, j\le 1} \|\Omega^\mu \nabla_r^j
  h\|_{L^2(\{|x|\in [R/2,4R]\})}.
\end{equation}
\end{lemma}

This lemma is proved by apply Sobolev embeddings on $\R\times S^3$
after localizing to the annulus.  The decay results from the
difference in the volume element for $\R\times S^3$ versus that of
$\R^4$ in polar coordinates.  See \cite{Klainerman}.

%%%%%%%%%%%%%%%%%%%%%%%%%%%%%%%%%%%%%%%%%%%%%%%%%%%%%%%%%%%%%%%%%%%%%%%%%%%%%%%%%%%%%%%%%%%%%%%%%%%%
%%%%%%%%%%%%%%%%%%%%%%%%%%%%%%%%%%%%%%%%%%%%%%%%%%%%%%%%%%%%%%%%%%%%%%%%%%%%%%%%%%%%%%%%%%%%%%%%%%%%
%%%%%%%%%%%%%%%%%%%%%%%%%%%%%%%%%%%%%%%%%%%%%%%%%%%%%%%%%%%%%%%%%%%%%%%%%%%%%%%%%%%%%%%%%%%%%%%%%%%%
\newsection{Main estimates on $\R_+\times\extfour$}
The main estimate here is a localized energy estimate which holds for
perturbations of the d'Alembertian.
The technique of proof
described in the previous section continues to hold.  The boundary
term that arises due to the obstacle, thanks to the star-shapedness
assumption, has a favorable sign and can simply be dropped.  

In order to handle the highest order terms of the quasilinear
equation, it is beneficial to have an analog of the localized energy
estimate for perturbations of the d'Alembertian.  We suppose that
$\phi\in C^\infty(\R_+\times\extfour)$ solves 
\begin{equation}
  \label{perturbedEquation}
\begin{cases}
\Box_h\phi = F,\quad (t,x)\in \R_+\times\extfour,\\
\phi|_\bdy=0,\\
\phi(0,\cd)=f,\quad \partial_t\phi(0,\cd)=g\end{cases}
\end{equation}
where $f,g$ are smooth, supported in $|x|<R$, and satisfy the
smallness condition \eqref{dataSmallness2}.  Here
\[\Box_h \phi = (\partial_t^2 - \Delta)\phi + 
h^{\alpha\beta}(t,x)\partial_\alpha\partial_\beta \phi,\]
and
\begin{equation}\label{sym}h^{\alpha\beta}(t,x)=h^{\beta\alpha}(t,x)\end{equation}
as well as
\begin{equation}\label{small}|h|=\sum_{\alpha,\beta=0}^4 |h^{\alpha\beta}(t,x)|\le
  \delta \ll 1.\end{equation}
We shall also utilize the notation
\[|\partial h|=\sum_{\alpha,\beta,\gamma=0}^4 |\partial_\gamma
h^{\alpha\beta}(t,x)|\]
and
\[S_T^\K = [0,T]\times \extfour.\]

For such a $\phi$, we
have the following estimates from \cite{MS}.
\begin{proposition}\label{kss.prop}
Suppose that $\K\subset\{|x|<1\}$ is a smooth, bounded, star-shaped obstacle as
above.  Let $\phi\in C^\infty(\R_+\times\extfour)$ solve
\eqref{perturbedEquation} with data satisfying \eqref{dataSmallness2}
and the compatibility conditions.  Suppose further that 
\[\sum_{|\mu|\le N} \|\partial^\mu \Box
\phi(0,\cd)\|_{L^2(\extfour)}\le C\varepsilon.\]  
Suppose that $h^{\alpha\beta}$ satisfies
\eqref{sym} and \eqref{small} for a sufficiently small choice of
$\delta$.  Then for any nonnegative integer $N$ and for any $T>0$,
\begin{multline}
  \label{kssext1}
%(\log(2+T))^{-1/2} \sum_{|\mu|\le N} \|\la x\ra^{-1/2} \partial^\mu
%  \phi'\|_{L^2(S^\K_T)}
%+ 
\sum_{|\mu|\le N} \|\la x\ra^{-1/2-}
  \partial^\mu \phi'\|_{L^2(S_T^\K)} %\\
+ \sum_{|\mu|\le N} \|\la
  x\ra^{-3/2} \partial^\mu \phi\|_{L^2(S_T^\K)} + \sum_{|\mu|\le N}
  \|\partial^\mu \phi'(T,\cd)\|_{L^2(\extfour)}
\\\lesssim \varepsilon + \sum_{j\le N} \int_0^T \|\Box_h \partial_t^j
\phi(t,\cd)\|_{L^2(\extfour)}\:dt
+ \sum_{j\le N} \int_0^T \Bigl\|\Bigl(|\partial h|+
  \frac{|h|}{r}\Bigr)|\nabla \partial^j_t\phi|\Bigr\|_{L^2(\extfour)}\:dt
\\+ \sum_{|\mu|\le N-1} \|\Box\partial^\mu \phi\|_{L^2(S_T^\K)} +
\sum_{|\mu|\le N-1} \|\Box \partial^\mu \phi(T,\cd)\|_{L^2(\extfour)},
\end{multline}
and
\begin{multline}
  \label{kssext2}
\sum_{|\mu|\le N} \|\la x\ra^{-1/2-} Z^\mu \phi'\|_{L^2(S^\K_T)} +
\sum_{|\mu|\le N} \|\la x\ra^{-3/2} Z^\mu \phi\|_{L^2(S^\K_T)} +
\sum_{|\mu|\le N} \|Z^\mu \phi'(T,\cd)\|_{L^2(\extfour)}
\\\lesssim \varepsilon + \sum_{|\mu|\le N} \int_0^T \|\Box_h Z^\mu
\phi(t,\cd)\|_2\:dt
+\sum_{|\mu|\le N} \int_0^T \Bigl\|\Bigl(|\partial h|+
  \frac{|h|}{r}\Bigr) |\nabla Z^\mu\phi|\Bigr\|_{L^2(\extfour)}\:dt
\\+ \sum_{|\mu|\le N+1} \|\partial^\mu_x \phi'\|_{L^2([0,T]\times \{|x|<1\})}.
\end{multline}
\end{proposition}

As stated above, the description of the proof in the previous sections
remains valid when $N=0$.  For the higher order cases, one uses the
fact that $\partial_t$ commutes with $\Box$ and preserves the
Dirichlet boundary conditions.  Then, by using elliptic regularity and
relating the Laplacian to time derivatives via the equation, one can
obtain \eqref{kssext1}.  To prove \eqref{kssext2}, one argues as in
the $N=0$ case and applies a trace theorem to the resulting boundary
terms.  Here we note that the coefficients of $Z$ are $O(1)$ on
$\bdy$.  Also note that the last term in \eqref{kssext2} is controlled
by the left side of \eqref{kssext1}.

\newsection{Proof of Theorem \ref{thm1.2}}
With the estimates of the previous sections in hand, we now proceed to
the proof of the main long time existence result.  

We solve the nonlinear equation via an iteration.  We set $u_0\equiv
0$ and recursively define $u_l$ to solve
\begin{equation}\label{iteration}
\begin{cases}
\Box u_l = a^\alpha u_{l-1} \partial_\alpha u_{l-1} +
b^{\alpha\beta}\partial_\alpha u_{l-1} \partial_\beta u_{l-1} + 
A^{\alpha\beta}
u_{l-1}\partial_\alpha\partial_\beta u_l + B^{\alpha\beta\gamma}\partial_\alpha u_{l-1}
\partial_\beta\partial_\gamma u_l\\
u_l|_\bdy=0\\
u_l(0,\cd)=f,\quad \partial_t u_l(0,\cd)=g.
\end{cases}
\end{equation}

For any fixed $0<\delta\le 1/8$, set
\begin{multline}
  \label{M}
M_l(T)=\sum_{|\mu|\le 50} \Bigl[\sup_{t\in [0,T]} \|(\partial^\mu
  u_l)'(t,\cd)\|_{L^2(\extfour)} + \|\la x\ra^{-1/2-\delta}
 (\partial^\mu u_l)'\|_{L^2(S_T^\K)}\Bigr] 
\\+ \sum_{|\mu|\le 49} \Bigl[\|\la x\ra^{-1/2-\delta} (Z^\mu
u_l)'\|_{L^2(S_T^\K)}
+  \|\la x\ra^{-1/2-2\delta} Z^\mu u_l\|_{L^2([0,T]\times \{|x|>2\})}
\\+ \sup_{t\in [0,T]} \|(Z^\mu u_l)'(t,\cd)\|_{L^2(\extfour)}
\Bigr]
+ \|u_l\|_{L^2([0,T]\times\{|x|<3\})}. 
\end{multline}
Our first goal is to show that $M_l(T)$ is bounded uniformly in $l$
and $T$.

We claim first that there is a $C_0$ so that $M_1(T)\le
C_0\varepsilon$ for any $T>0$.  Such boundedness follows easily from
Proposition \ref{kss.prop} (with $h^{\alpha\beta}=0$) and \eqref{dataSmallness2} for every term
in $M_1(T)$ except for the fourth term.  For this fourth term, we fix
a smooth test function $\rho$ which is identically $1$ on $\{|x|<1\}$
and vanishes outside of
$\{|x|<2\}$.  Since $\K\subset\{|x|<1\}$, we then have that $(1-\rho)Z^\mu u_1$ solves the
boundaryless wave equation
\[\Box (1-\rho)Z^\mu u_1 = [\Delta,\rho]Z^\mu u_1.\]
Since the commutator has compact support, we may apply
\eqref{weightedStrichartz} (with $F=0$) and \eqref{cpctF}.  What
results in \eqref{cpctF} from the commutator is easily bounded
using \eqref{kssext1} as above.

We now inductively show that 
\[M_l(T) \le 10 C_0 \varepsilon,\quad l=2,3,\dots.\]
We begin with bounding everything but the fourth term of $M_l(T)$
using Proposition \ref{kss.prop}.  Here, we set
\[h^{\alpha\beta}=-A^{\alpha\beta}u_{l-1} -
B^{\gamma\alpha\beta}\partial_\gamma u_{l-1}.\]
It then follows that terms $I$, $II$, $III$, $V$, and $VI$ of \eqref{M} are 
\begin{multline}\label{kssPieces}
  \le C_0\varepsilon + C\sum_{|\mu|\le 50} \int_0^T \|\partial^\mu
  \Box_h u_l(t,\cd)\|_{L^2(\extfour)}\:dt
+ C\sum_{|\mu|\le 50} \int_0^T
\|[\Box_h,\partial^\mu]u_l(t,\cd)\|_{L^2(\extfour)}\:dt
\\+ C\sum_{\substack{|\mu|\le 50\\|\sigma|\le 2}}
\int_0^T \|(\partial^\sigma u_{l-1}) (
  \partial^\mu u_l')\|_{L^2(\extfour)}\:dt
\\+  C\sum_{|\mu|\le 49} \int_0^T \|Z^\mu
  \Box_h u_l(t,\cd)\|_{L^2(\extfour)}\:dt
+ C\sum_{|\mu|\le 49} \int_0^T
\|[\Box_h,Z^\mu]u_l(t,\cd)\|_{L^2(\extfour)}\:dt
\\+ C\sum_{\substack{|\mu|\le 49\\|\sigma|\le 2}}
\int_0^T \|(Z^\sigma u_{l-1}) (
  Z^\mu u_l')\|_{L^2(\extfour)} \:dt
\\+C\sum_{|\mu|\le 49} \Bigl[\sup_{t\in [0,T]} \|\partial^\mu \Box
  u_l(t,\cd)\|_{L^2(\extfour)} + \|\partial^\mu \Box u_l\|_{L^2(S_T^\K)}\Bigr].
\end{multline}
Noting that
\begin{multline*}
  \sum_{|\mu|\le 50}\Bigl(|\partial^\mu \Box_h
  u_l|+|[\partial^\mu,\Box_h] u_l|\Bigr)
\lesssim \sum_{|\mu|\le 26} |\partial^\mu u_{l-1}| \sum_{|\nu|\le 50}
|\partial^\nu u_l'| \\+ \sum_{|\mu|\le 27} |\partial^\mu u_l|
\sum_{|\nu|\le 50} |\partial^\nu u_{l-1}'|
+ \sum_{|\mu|\le 26} |\partial^\mu u_{l-1}| \sum_{|\nu|\le 50}
|\partial^\nu u_{l-1}'|
\end{multline*}
and similarly
\begin{multline*}
  \sum_{|\mu|\le 49}\Bigl(|Z^\mu \Box_h
  u_l|+|[Z^\mu,\Box_h] u_l|\Bigr)
\lesssim \sum_{|\mu|\le 26} |Z^\mu u_{l-1}| \sum_{|\nu|\le
    49}
|Z^\nu u_l'| \\+ \sum_{|\mu|\le 27} |Z^\mu u_l|
\sum_{\substack{|\nu|\le 49\\|\sigma|\le 1}} |Z^\nu
\partial^\sigma u_{l-1}|
+ \sum_{|\mu|\le 26} |Z^\mu u_{l-1}| \sum_{\substack{|\nu|\le
    49\\|\sigma|\le 1}}
|Z^\nu \partial^\sigma u_{l-1}|,
\end{multline*}
we may apply \eqref{weightedSobolev} and the Schwarz inequality to
control the second through seventh terms in \eqref{kssPieces} by
\begin{multline*}
  \sum_{|\mu|\le 29} \|\la x\ra^{-3/4} Z^\mu u_{l-1}\|_{L^2(S_T^\K)}
  \Bigl(\sum_{|\nu|\le 50} \|\la x\ra^{-3/4} \partial^\nu
  u_l'\|_{L^2(S_T^\K)} + \sum_{|\nu|\le 49}
  \|\la x\ra^{-3/4} Z^\nu u_l'\|_{L^2(S_T^\K)}\Bigr)
\\+ \sum_{|\mu|\le 30} \|\la x\ra^{-3/4} Z^\mu u_{l}\|_{L^2(S_T^\K)}
  \Bigl(\sum_{|\nu|\le 50} \|\la x\ra^{-3/4} \partial^\nu
  u_{l-1}'\|_{L^2(S_T^\K)} + \sum_{\substack{|\nu|\le 49\\|\sigma|\le 1}}
  \|\la x\ra^{-3/4} Z^\nu\partial^\sigma u_{l-1}\|_{L^2(S_T^\K)}\Bigr)
\\+\sum_{|\mu|\le 29} \|\la x\ra^{-3/4} Z^\mu u_{l-1}\|_{L^2(S_T^\K)}
  \Bigl(\sum_{|\nu|\le 50} \|\la x\ra^{-3/4} \partial^\nu
  u_{l-1}'\|_{L^2(S_T^\K)} + \sum_{\substack{|\nu|\le 49\\|\sigma|\le 1}}
  \|\la x\ra^{-3/4} Z^\nu\partial^\sigma u_{l-1}\|_{L^2(S_T^\K)}\Bigr).
\end{multline*}

The last two terms of \eqref{kssPieces} are controlled similarly, though
more simply.  Here, for the pointwise in time terms, it is worth
noting the following variant of \eqref{weightedSobolev}.
\begin{proposition}
If $h\in C^\infty_0(\R^4)$ and $R>0$, then
\[R \|h\|_{L^\infty(\{R/2<|x|<R\})}\lesssim \sum_{|\mu|\le 2}
\|\nabla_x \Omega^\mu h\|_{L^2(\R^4)}.\]  
\end{proposition}
See, e.g., \cite[Lemma 2.1]{DMSZ}.
It is also necessary to invoke a Hardy inequality.  These details are
left to the reader.

With this, we have seen that terms $I$, $II$, $III$, $V$, and $VI$ of
\eqref{M} are
\[\le C_0\varepsilon + CM_{l-1}(T)M_l(T) + C (M_{l-1}(T))^2.\]

It remains only to bound term $IV$ of \eqref{M}.  For this, for a
solution $u_l$ to \eqref{iteration}, we examine $(1-\rho)Z^\mu u_l$
where $|\mu|\le 49$ and $\rho$ is a smooth cutoff function as above.
We have that $(1-\rho)Z^\mu u_l$ solves
the boundaryless wave equation
\begin{multline*}
  \Box (1-\rho)Z^\mu u_l = [\Delta,\rho]Z^\mu u_l 
+  (1-\rho)\sum_{|\sigma|+|\nu|\le |\mu|} a^\alpha_{\mu\nu\sigma} \partial_\alpha (Z^\nu
  u_{l-1} Z^\sigma u_{l-1})
\\+ (1-\rho)\sum_{|\sigma|+|\nu|\le|\mu|} b^{\alpha\beta}_{\mu\nu\sigma}
\partial_\alpha Z^\sigma u_{l-1} \partial_\beta Z^\nu u_{l-1} 
+ (1-\rho)\sum_{|\sigma|+|\nu|\le |\mu|} A^{\alpha\beta}_{\mu\nu\sigma}
\partial_\alpha (Z^\sigma u_{l-1} \partial_\beta Z^\nu u_l)
\\+ (1-\rho)\sum_{|\sigma|+|\nu|\le |\mu|}
\tilde{A}^{\alpha\beta}_{\mu\nu\sigma}  \partial_\alpha Z^\sigma u_{l-1} \partial_\beta
Z^\nu u_l
+ (1-\rho) \sum_{\substack{|\sigma|+|\nu|\le
    |\mu|\\|\sigma|>0}}B^{\alpha\beta\gamma}_{\mu\nu\sigma} \partial_\alpha Z^\sigma
u_{l-1} \partial_\beta\partial_\gamma Z^\nu u_l
\\+ (1-\rho) B^{\alpha\beta\gamma}_{\mu\mu 0} \partial_\beta (\partial_\alpha u_{l-1}
\partial_\gamma Z^\mu u_l)
+ (1-\rho) \tilde{B}^{\alpha\beta\gamma}_{\mu\mu 0} \partial_\alpha\partial_\beta u_{l-1}
\partial_\gamma Z^\mu u_l.
\end{multline*}
for appropriate coefficients $a^\alpha_{\mu\nu\sigma}$,
$b^{\alpha\beta}_{\mu\nu\sigma}$, $A^{\alpha\beta}_{\mu\nu\sigma}$,
etc. based on those in \eqref{trucatedEquation}.
Here we have used the fact that $[\partial,Z]$ is in the span of
$\{\partial\}$, and the new coefficients depend on these commutators
as well as the appropriate binomial coefficients.
We further decompose the right side into
\begin{multline*}
  [\Delta,\rho]Z^\mu u_l 
+ \sum_{|\sigma|+|\nu|\le |\mu|} a^\alpha_{\mu\nu\sigma}
\partial_\alpha ((1-\rho) Z^\mu
  u_{l-1} Z^\sigma u_{l-1})
\\+ \sum_{|\sigma|+|\nu|\le |\mu|} A^{\alpha\beta}_{\mu\nu\sigma}
\partial_\alpha ((1-\rho) Z^\sigma u_{l-1} \partial_\beta Z^\nu u_l)
+ B^{\alpha\beta\gamma}_{\mu\mu 0} \partial_\beta ((1-\rho)\partial_\alpha u_{l-1}
\partial_\gamma Z^\mu u_l)
\\+ (1-\rho)\sum_{|\sigma|+|\nu|\le|\mu|} b^{\alpha\beta}_{\mu\nu\sigma}
\partial_\alpha Z^\sigma u_{l-1} \partial_\beta Z^\nu u_{l-1} 
+ (1-\rho)\sum_{|\sigma|+|\nu|\le |\mu|}
\tilde{A}^{\alpha\beta}_{\mu\nu\sigma}  \partial_\alpha Z^\sigma u_{l-1} \partial_\beta
Z^\nu u_l
\\+ (1-\rho) \sum_{\substack{|\sigma|+|\nu|\le
    |\mu|\\|\sigma|>0}}B^{\alpha\beta\gamma}_{\mu\nu\sigma} \partial_\alpha Z^\sigma
u_{l-1} \partial_\beta\partial_\gamma Z^\nu u_l
+ (1-\rho) \tilde{B}^{\alpha\beta\gamma}_{\mu\mu 0} \partial_\alpha\partial_\beta u_{l-1}
\partial_\gamma Z^\mu u_l
\\+
\sum_{|\sigma|+|\nu|\le|\mu|}C^{\alpha}_{\mu\nu\sigma}(\partial_\alpha\rho)Z^\nu
u_{l-1} Z^\sigma u_{l-1} + \sum_{\substack{|\sigma|+|\nu|\le
    |\mu|\\\delta \le
    1}}\tilde{C}^{\alpha\beta\gamma}_{\mu\nu\sigma\delta}(\partial_\alpha\rho)
  \partial_\beta^\delta Z^\sigma u_{l-1} \partial_\gamma Z^\nu u_l.
\end{multline*}
We write $(1-\rho)Z^\mu u_l = v_1 + v_2 + v_3$ where $v_1$ solves
$\Box v_1 = [\Delta,\rho]Z^\mu u_l$ and
\begin{multline*}\Box v_2 =  \sum_{|\sigma|+|\nu|\le |\mu|} a^\alpha_{\mu\nu\sigma}
\partial_\alpha ((1-\rho) Z^\nu
  u_{l-1} Z^\sigma u_{l-1})
\\+ \sum_{|\sigma|+|\nu|\le |\mu|} A^{\alpha\beta}_{\mu\nu\sigma}
\partial_\alpha ((1-\rho) Z^\sigma u_{l-1} \partial_\beta Z^\mu u_l)
+ B^{\alpha\beta\gamma}_{\mu\mu 0} \partial_\beta ((1-\rho)\partial_\alpha u_{l-1}
\partial_\gamma Z^\nu u_l).\end{multline*}
Both $v_1$ and $v_2$ are taken to have vanishing initial data.  It
then follows that $v_3$ solves
\begin{multline*}
  \Box v_3=(1-\rho)\sum_{|\sigma|+|\nu|\le|\mu|} b^{\alpha\beta}_{\mu\nu\sigma}
\partial_\alpha Z^\sigma u_{l-1} \partial_\beta Z^\nu u_{l-1} 
+ (1-\rho)\sum_{|\sigma|+|\nu|\le |\mu|}
\tilde{A}^{\alpha\beta}_{\mu\nu\sigma}  \partial_\alpha Z^\sigma u_{l-1} \partial_\beta
Z^\nu u_l
\\+ (1-\rho) \sum_{\substack{|\sigma|+|\nu|\le
    |\mu|\\|\sigma|>0}}B^{\alpha\beta\gamma}_{\mu\nu\sigma} \partial_\alpha Z^\sigma
u_{l-1} \partial_\beta\partial_\gamma Z^\nu u_l
+ (1-\rho) \tilde{B}^{\alpha\beta\gamma}_{\mu\mu 0} \partial_\alpha\partial_\beta u_{l-1}
\partial_\gamma Z^\mu u_l
\\+
\sum_{|\sigma|+|\nu|\le|\mu|}C^{\alpha}_{\mu\nu\sigma}(\partial_\alpha\rho)Z^\nu
u_{l-1} Z^\sigma u_{l-1} + \sum_{\substack{|\sigma|+|\nu|\le
    |\mu|\\\delta \le
    1}}\tilde{C}^{\alpha\beta\gamma}_{\mu\nu\sigma\delta}(\partial_\alpha\rho)
  \partial_\beta^\delta Z^\sigma u_{l-1} \partial_\gamma Z^\nu u_l
\end{multline*}
with Cauchy data which matches that of $(1-\rho)Z^\mu u_l$.

To bound 
\[\|\la x\ra^{-1/2-2\delta} v_1\|_{L^2([0,T]\times \R^4)},\] 
we apply \eqref{cpctF}.  It then remains to bound 
\[\|Z^\mu
u_l\|_{L^2([0,T]\times \{1<|x|<2\})} + \|(Z^\mu
u_l)'\|_{L^2([0,T]\times \{1<|x|<2\})}\] 
which we have done in the
previous step.

To establish a bound for $v_2$ in the same space, we may apply
\eqref{lindbladVariant}.  Using the compatibility conditions and
\eqref{dataSmallness2}, we have
\begin{multline*}
  \|\la x\ra^{-1/2-2\delta} v_2\|_{L^2([0,T]\times\R^4)}\le C_0
  \varepsilon + C\sum_{|\nu|\le 49, |\sigma|\le 25}\int_0^T
  \|Z^\sigma u_{l-1} Z^\nu u_{l-1}\|_{L^2(\extfour)} \:dt
\\+ C\sum_{|\nu|\le 49, |\sigma|\le 25} \int_0^T \|Z^\sigma u_{l-1}
\partial Z^\nu u_l\|_{L^2(\extfour)}\:dt
\\+ C \sum_{|\nu|\le 49, |\sigma|\le 25} \int_0^T \|Z^\sigma \partial
u_l Z^\nu u_{l-1}\|_{L^2(\extfour)}\:dt.
\end{multline*}
For each of the latter three terms, we apply \eqref{weightedSobolev}
to the lower order term on each dyadic annulus.  By applying the
Schwarz inequality and summing over the dyadic intervals, we see that
this is
\[\le C_0 \varepsilon + C (M_{l-1}(T))^2 + C M_{l-1}(T) M_l(T)\]
provided $\delta\le 1/8$.

For $v_3$, we use \eqref{weightedStrichartz}.  Using, again, the
compatibility conditions and \eqref{dataSmallness2}, we have
\begin{multline*}
  \|\la x\ra^{-1/2-2\delta} v_3\|_{L^2([0,T]\times\R^4)} \le
  C_0\varepsilon + C\sum_{|\nu|\le 49,|\sigma|\le 25}
\|\la x\ra^{-1-2\delta} (\partial Z^\sigma u_{l-1})(\partial Z^\nu u_{l-1})
\|_{L^1_tL^1_rL^2_\omega}
\\+ C\sum_{|\nu|\le 49,|\sigma|\le 25}
\|\la x\ra^{-1-2\delta} (\partial Z^\sigma u_{l-1})(\partial Z^\nu u_{l})
\|_{L^1_tL^1_rL^2_\omega}
\\+ C\sum_{|\nu|\le 49,|\sigma|\le 25}
\|\la x\ra^{-1-2\delta} (\partial Z^\sigma u_{l})(\partial Z^\nu
u_{l-1}) \|_{L^1_tL^1_rL^2_\omega}
\\+C\sum_{|\nu|\le 49,|\sigma|\le 25} \|Z^\sigma u_{l-1} Z^\nu
u_{l-1}\|_{L^1_tL^1_rL^2_\omega([0,T]\times \{1<|x|<2\})} 
\\+ C\sum_{|\nu|\le 49,|\sigma|\le 26}
\|Z^\sigma u_{l-1} \partial Z^\nu
u_l\|_{L^1_tL^1_rL^2_\omega([0,T]\times \{1<|x|<2\})}
\\+ C\sum_{\substack{|\nu|\le 49,|\sigma|\le 26\\\delta\le 1}} \|Z^\sigma
u_l \partial^\delta Z^\nu u_{l-1}\|_{L^1_tL^1_rL^2_\omega([0,T]\times\{1<|x|<2\})}.
\end{multline*}
By applying Sobolev embeddings on $S^3$ to the lower order term in
each of the last six terms and utilizing the Schwarz inequality,  
we again have that this is controlled by
\[C_0\varepsilon + C(M_{l-1}(T))^2 + C M_{l-1}(T) M_l(T).\]

By combining the bounds for each of these pieces, we see that
\[M_l(T)\le 4C_0\varepsilon + C(M_{l-1}(T))^2 + C M_{l-1}(T)M_l(T).\]
If we apply the inductive hypothesis $M_{l-1}(T)\le 10
C_0\varepsilon$, it indeed follows that
\[M_l(T)\le 10 C_0 \varepsilon\]
provided that $\varepsilon$ is sufficiently small.

It remains to show that $\{u_l\}$ is a Cauchy sequence in similar
spaces.  To this end, we set
\begin{multline}
  \label{A}
A_l(T)=\sum_{|\mu|\le 49} \Bigl[\sup_{t\in [0,T]} \|(\partial^\mu
  (u_l-u_{l-1}))'(t,\cd)\|_{L^2(\extfour)} + \|\la x\ra^{-1/2-\delta}
 (\partial^\mu (u_l-u_{l-1}))'\|_{L^2(S_T^\K)}\Bigr] 
\\+ \sum_{|\mu|\le 48} \Bigl[\|\la x\ra^{-1/2-\delta} (Z^\mu
(u_l-u_{l-1}))'\|_{L^2(S_T^\K)}
+  \|\la x\ra^{-1/2-2\delta} Z^\mu (u_l-u_{l-1})\|_{L^2([0,T]\times \{|x|>2\})}
\\+ \sup_{t\in [0,T]} \|(Z^\mu (u_l-u_{l-1}))'(t,\cd)\|_{L^2(\extfour)}
\Bigr]
+ \|u_l-u_{l-1}\|_{L^2([0,T]\times\{|x|<3\})}. 
\end{multline}
Using quite similar arguments and the $O(\varepsilon)$ bound on
$M_l(T)$, one can prove
\[A_l(T)\le \frac{1}{2}A_{l-1}(T).\]
This suffices to show that the sequence converges.  Its limit is
indeed the solution $u$, which completes the proof.

%%%%%%%%%%%%%%%%%%%%%%%%%%%%%%%%%%%%%%%%%%%%%%%%%%%%%%%%%%%%%%%%%%%%%%%%%%%%%%%%%%%%%%%%%%%%%%%%%%%%
%%%%%%%%%%%%%%%%%%%%%%%%%%%%%%%%%%%%%%%%%%%%%%%%%%%%%%%%%%%%%%%%%%%%%%%%%%%%%%%%%%%%%%%%%%%%%%%%%%%%
%%%%%%%%%%%%%%%%%%%%%%%%%%%%%%%%%%%%%%%%%%%%%%%%%%%%%%%%%%%%%%%%%%%%%%%%%%%%%%%%%%%%%%%%%%%%%%%%%%%%


\begin{thebibliography}{MA}
%\bibitem{Ben} M. Ben-Artzi: 
%{\em Regularity and smoothing for some equations of evolution}, in 
%"Nonliner Partial Differential Equations and Applications" 
%(H. Brezis and J. L. Lions, eds.), London, 1994, pp. 1--12.

%\bibitem{BenK} M. Ben-Artzi and S. Klainerman: 
%{\em Decay and regularity for the Schr\"odinger equation}, 
%J. Anal. Math. {\bf 58} (1992), 25--37.

%\bibitem{BSS} M. Blair, H. Smith and C. D. Sogge: 
%{\em Strichartz estimates for the wave equation on manifolds with boundary}, 
%arXiv:0805.4733.

%\bibitem{burq} N. Burq: {\em Global Strichartz estimates for 
%nontrapping geometries: About an article by H. Smith and C. Sogge}, 
%Comm. Partial Differential Equations  {\bf 28} (2003), 1675--1683.
 
%\bibitem{BLP} N. Burq, G. Lebeau and F. Planchon: 
%{\em Global existence for energy critical waves in 3-D domains}, 
%J. Amer. Math. Soc. {\bf 21} (2008), 831--845.

%\bibitem{BP} N. Burq and F. Planchon: 
%{\em Global existence for energy critical waves in 3-d domains : 
%Neumann boundary conditions}, arXiv:0711.0275.

%\bibitem{CK} M. Christ and A. Kiselev: 
%{\em Maximal functions associated to filtrations}, 
%J. Funct. Anal. {\bf 179} (2001), 409--425.

\bibitem{DMSZ}Y. Du, J. Metcalfe, C. D. Sogge and Y. Zhou: 
{\em Concerning the Strauss conjecture and almost global existence 
for nonlinear Dirichlet-wave equations in $4$-dimensions},
Comm. Partial Differential Equations {\bf 33} (2008), 1487--1506.

\bibitem{DZ} Y. Du and Y. Zhou: 
{\em The life span for nonlinear wave equation outside of star-shaped
obstacle in three space dimensions}, Comm. Partial Differential
Equations {\bf 33} (2008), 1455--1486.

\bibitem{FW} D. Fang and C. Wang: 
{\em Weighted Strichartz Estimates with Angular Regularity and their 
Applications}, arXiv:0802.0058.

%\bibitem{GLS} V. Georgiev, H. Lindblad, and C. D. Sogge: 
%{\em Weighted Strichartz estimates and global existence for 
%semilinear wave equations},  
%Amer. J. Math.  {\bf 119}  (1997),  1291--1319.

%\bibitem{Glassey} R. T. Glassey: 
%{\em Existence in the large for $\square u = F(u)$ in two dimensions}, 
%Math. Z. {\bf 178} (1981), 233--261.

%\bibitem{H} K. Hidano: 
%{\em Morawetz-Strichartz estimates for spherically symmetric solutions 
%to wave equations and applications to semi-linear Cauchy problems}, 
%Differential Integral Equations {\bf 20} (2007), 735--754.

%\bibitem{H2}  K. Hidano: 
%{\em Small solutions to semi-linear wave equations with radial data 
%of critical regularity}, Rev. Mat. Iberoamericana, to appear.

\bibitem{HMSSZ} K. Hidano, J. Metcalfe, H. F. Smith, C. D. Sogge,
  Y. Zhou: {\em On abstract Strichartz estimates and the Strauss
    conjecture for nontrapping obstacles}.  Trans. Amer. Math. Soc.,
  to appear.

\bibitem{Hormander} L. H\"ormander: {\em On the fully nonlinear Cauchy
  problem with small data. II.}  Microlocal analysis and nonlinear
  waves (Minneapolis, MN, 1988--1989), 51--81, IMA Vol. Math. Appl.,
  30, Springer, New York, 1991.

%\bibitem{Hos}{T. Hoshiro}: 
%{\em On weighted $L^2$ estimates of solutions to wave equations}, 
%J. Anal. Math. {\bf 72} (1997), 127--140.

%\bibitem{John} F. John: 
%{\em Blow-up of solutions of nonlinear wave equations in three space 
%dimensions},  
%Manuscripta Math. {\bf 28} (1979), 235--265

\bibitem{KSS}  M. Keel, H. F. Smith, and C. D. Sogge: 
{\em Almost global existence for some semilinear wave equations}, 
J. Anal. Math. {\bf 87} (2002), 265-279.

\bibitem{KSS2} M. Keel, H. F. Smith, and C. D. Sogge: {\em Global
  existence for a quasilinear wave equation outside of star-shaped
  domains}.  J. Funct. Anal. {\bf 189} (2002), 155--226.

%\bibitem{KT} M. Keel and T. Tao, 
%{\em Endpoint Strichartz estimates}, 
%Amer. J. Math. {\bf 120} (1998), 955--980.

\bibitem{Klainerman} S. Klainerman: {\em The null condition and global
  existence to nonlinear wave equatoins}.  Lect. Appl. Math. {\bf 23}
  (1986), 293--326.

\bibitem{Koch} H. Koch: {\em Mixed problems for fully nonlinear
  hyperbolic equations}.  Math. Z.  {\bf 214} (1993), 9--42.

%\bibitem{levine} H. A. Levine, {\em Instability and nonexistence of
%global solutions to 
%nonlinear wave equations of the form} $Pu\sb{tt}=-Au+{\cal F}(u)$, 
%Trans. Amer. Math. Soc. {\bf 192} (1974), 1--21.

%\bibitem{LX}{T. T. Li and Y. Zhou}: 
%{\em A note on the life-span of classical solutions to nonlinear wave 
%equations in four space dimensions}, 
%Indiana Univ. Math. J. {\bf 44} (1995), 1207--1248.

\bibitem{Lindblad} H. Lindblad: {\em On the lifespan of solutions of
  nonlinear wave equations with small initial data}.  Comm. Pure
  Appl. Math. {\bf 43} (1990), 445--472.

%\bibitem{LS}{H. Lindblad and C. D. Sogge}: 
%{\em On existence and scattering with minimal regularity for semilinear 
%wave equations}, 
%J. Funct. Anal. {\bf 130} (1995), 357--426.

%\bibitem{LS2}{H. Lindblad and C. D. Sogge}: 
%{\em Long-time existence for small amplitude semilinear wave equations}, 
%Amer. J. Math. {\bf 118} (1996), 1047--1135.

%\bibitem{Mel} R. B. Melrose: 
%{\em Singularities and energy decay in acoustical scattering}, 
%Duke Math. J. {\bf 46} (1979), 43--59.

%\bibitem{MelSj} R. B. Melrose and J. Sj\"ostrand: 
%{\em Singularities of boundary value problems. I}, 
%Comm. Pure Appl. Math. {\bf 31} (1978), 593--617.

%\bibitem{M}{J. Metcalfe}: 
%{\em Global Strichartz estimates for solutions to the wave equation 
%exterior to a convex obstacle},   
%Trans. Amer. Math. Soc.  {\bf 356}, (2004), 4839--4855.

\bibitem{MS3} J. Metcalfe and C. D. Sogge: {\em Hyperbolic trapped
  rays and global existence of quasilinear wave equations}.
  Invent. Math. {\bf 159} (2005), 75--117.

\bibitem{MS} J. Metcalfe and C. D. Sogge: {\em Long time existence of
  quasilinear wave equations exterior to star-shaped obstacles via
  energy methods}.  SIAM J. Math. Anal. {\bf 38} (2006), 188--209.

\bibitem{MS2} J. Metcalfe and C. D. Sogge: {\em Global existence for
  Dirichlet-wave equations with quadratic nonlinearities in high
  dimensions}.  Math. Ann. {\bf 336} (2006), 391--420.

\bibitem{MSS} J. Metcalfe, C. D. Sogge, and A. Stewart: {\em Nonlinear
  hyperbolic equations in infinite homogeneous waveguides}.
  Comm. Partial Differential Equations {\bf 30} (2005), 643--661.

%\bibitem{mor} C. S. Morawetz: 
%{\em Decay for solutions of the exterior problem for the wave equation}, 
%Comm. Pure and Appl. Math.  {\bf 28} (1975), 229--264.

%\bibitem{mrs} C. S. Morawetz, J. Ralston and W. Strauss: 
%{\em Decay of solutions of the wave equation outside nontrapping obstacles},
%Comm. Pure Appl. Math. {\bf 30} (1977), 87--133.

%\bibitem{Lax} P. D. Lax and R. S. Philips: 
%Scattering Theory (Revised Edition), Academic Press Inc., 1989.

%\bibitem{Ral} J. Ralston: 
%{\em Note on the decay of acoustic waves}, 
%Duke Math. J. {\bf 46} (1979), 799-804.

%\bibitem{ST} Y. Shibata and Y. Tsutsumi: 
%{\em Global existence theorem for nonlinear wave equation in exterior domain},
%Lecture Notes in Num. Appl. Anal., Vol. 6 (1983), 155-196.

%\bibitem{Sideris} T. C. Sideris: 
%{\em Nonexistence of global solutions to semilinear wave equations 
%in high dimensions}, 
%J. Differential Equations {\bf 52} (1984), 378--406.

%\bibitem{SS0} H. F. Smith and C. D. Sogge: 
%{\em On the critical semilinear wave equation outside convex obstacles}, 
%J. Amer. Math. Soc. {\bf 8} (1995), 879--916.

%\bibitem{SS} H. F. Smith and C. D. Sogge: 
%{\em Global Strichartz estimates for nontrapping perturbations of 
%the Laplacian}, 
%Comm. Partial Differential Equations {\bf 25}, (2000), 2171--2183.

%\bibitem{SS2} H. F. Smith and C. D. Sogge: 
%{\em On the $L\sp p$ norm of spectral clusters for compact manifolds 
%with boundary}, Acta Math. {\bf 198}, (2007), 107--153.

%\bibitem{So} C. D. Sogge: 
%Lectures on nonlinear wave equations, International Press, Boston, MA 1995.

%\bibitem{So2} C. D. Sogge: 
%Lectures on nonlinear wave equations, 2nd edition, 
%International Press, Boston, MA, 2008.

%\bibitem{Ta} D. Tataru: 
%{\em Strichartz estimates in the hyperbolic space and global existence 
%for the semilinear wave equation},  
%Trans. Amer. Math. Soc.  {\bf 353}  (2001),   795--807.

%\bibitem{taylor} M. Taylor:
%{\em Grazing rays and reflection of singularities of solutions
%to wave equations},
%Comm. Pure Appl. Math. {\bf 29} (1976), 1--38.

%\bibitem{V} B. R. Vainberg: 
%{\em The short-wave asymptotic behavior of the solutions of stationary 
%problems, and the asymptotic behavior as $t\rightarrow \infty $ of the 
%solutions of nonstationary problems}, 
%Russian Math. Surveys {\bf 30} (1975), 1--58.

%\bibitem{Zhou} Y. Zhou: 
%{\em Cauchy problem for semilinear wave equations with small data in 
%four space dimensions},
%J. Partial Differential Equations {\bf 8} (1995), 135--144.


\end{thebibliography}
\end{document}